\documentclass{amsart}
\usepackage{graphicx}
\usepackage{bm}
\usepackage{color}
\usepackage{amssymb, amsmath, amsthm}
\usepackage{bm}
\theoremstyle{plain}
\usepackage{graphicx}
\usepackage[aboveskip=-5pt,position=bottom]{caption}
\usepackage{natbib}
\usepackage{url}
\usepackage[plainpages=false,pdfpagelabels,breaklinks,linktocpage,hyperfootnotes=false]{hyperref}

\newtheorem{lemma}{Lemma}[section]

\newtheorem{theorem}[lemma]{Theorem}

\newtheorem{proposition}[lemma]{Proposition}

\newtheorem{remark}[lemma]{\normalfont \scshape Remark}
\newtheorem{definition}[lemma]{Definition}

\newcommand{\R}{\mathbb{R}}
\newcommand{\N}{\mathbb{N}}

\newcommand{\norm}[1]{\left\Vert#1\right\Vert}
\newcommand{\abs}[1]{\left\vert#1\right\vert}
\newcommand{\set}[1]{\left\{#1\right\}}

\newcommand{\eps}{\varepsilon}

\newcommand{\bfU}{\bm{U}}

\newcommand{\bfV}{\bm{V}}

\newcommand{\bfX}{\bm{X}}

\newcommand{\bfZ}{\bm{Z}}

\newcommand{\bfeta}{\bm{\eta}}

\newcommand{\bfxi}{\bm{\xi}}

\newcommand{\barE}{\bar E^-[0,1]}

\makeatletter
    \newcommand{\unit}{\@ifstar{\unit@Star}{\unit@NoStar}}
        \newcommand{\unit@Star}{(0,1)}
        \newcommand{\unit@NoStar}{[0,1]}
    \newcommand{\tab}{\@ifstar{\tab@Star}{\tab@NoStar}}
        \newcommand{\tab@Star}[1]{{}\mathrel{\phantom{#1}}}
        \newcommand{\tab@NoStar}[1]{\phantom{{}#1}}
\makeatother

\newcommand{\PinfZdz}{\,\left(P*\inf_{0\le t\le 1}Z_t\right)(dz)}

\allowdisplaybreaks[4]
\begin{document}

\title[Testing for a Generalized Pareto Process]{Testing for a Generalized Pareto Process}%
\author{Stefan Aulbach and Michael Falk}
\address{$^1$University of W\"{u}rzburg,
Institute of Mathematics,  Emil-Fischer-Str. 30, 97074 W\"{u}rzburg, Germany.}
\email{stefan.aulbach@uni-wuerzburg.de, falk@mathematik.uni-wuerzburg.de}


\begin{abstract}
We investigate two models for the following setup: We consider a stochastic process
$\bfX\in C[0,1]$ whose distribution belongs to a parametric family indexed by
$\vartheta\in\Theta\subset \R$. In case $\vartheta=0$,  $\bfX$ is a generalized Pareto
process. Based on $n$ independent copies $\bfX^{(1)},\dots,\bfX^{(n)}$ of $\bfX$, we
establish local asymptotic normality (LAN) of the point process of exceedances among
$\bfX^{(1)},\dots,\bfX^{(n)}$ above an increasing threshold line in each model.

The corresponding central sequences provide asymptotically optimal sequences of tests for
testing $H_0: \vartheta=0$ against a sequence of alternatives $H_n:
\vartheta=\vartheta_n$ converging to zero as $n$ increases. In one model, with an
underlying exponential family, the central sequence is provided by the number of
exceedances only, whereas in the other one the exceedances themselves contribute, too.
However it turns out that, in both cases, the test statistics also depend on some additional
and usually unknown model parameters.

We, therefore, consider an omnibus test statistic sequence as well and compute its
asymptotic relative efficiency with respect to the optimal test sequence.
\end{abstract}

\maketitle

\section{Introduction}

In the recent three decades, the focus of univariate extreme value theory shifted from the
investigation of maxima (minima) in a sample to the investigation of exceedances above a
high threshold. This approach towards large observations eased accessing the field of
extreme value theory and became a crucial tool for various applied disciplines, such as
building dykes.

Since the publications of the articles by \citet{balh74} and \citet{pick75} it is known that
exceedances above a high threshold can reasonably be modeled only by (univariate)
generalized Pareto distributions (GPD), resulting in the peaks-over-threshold approach
(POT).

Due to practical necessity, the focus of extreme value theory moved in recent years to
multivariate observations as well. Accordingly, the investigation of multivariate exceedances
enforced the definition and investigation of multivariate GPD. This investigation is still lively
continuing as even the definition of multivariate GPD is under debate; see, for instance,
\citet[Chapter 5]{fahure10}.

As already mentioned by \citet[p.\,293]{dehaf06}: \textit{Infinite-di\-men\-sional extreme
value theory is not just a theoretical extension of multivariate extreme value theory to a
more abstract context. It serves to solve concrete problems as well.} Such concrete
problems are, e.g., observing dykes and tides along their whole width and not only at a finite
set of observation points. There is, consequently, the need for a POT approach for
functional data and for generalized Pareto \emph{processes} as well. Again, the data
exceeding some kind of a high threshold are modeled by a functional counterpart of a GPD;
see \citet{aulfaho11b}. The current paper deals with optimal tests that check for particular models whether
those exceedances do, in fact, arise from such a corresponding process.

\subsection{Previous work}

Following \citet{buihz08}, a standard generalized Pareto process, i.e., a generalized Pareto
process with ultimately uniform tails in the margins, is defined as follows. For convenience,
we use bold font such as $\bfV$ for stochastic processes and default font such as $f$ for
non stochastic functions. All operations on functions such as $f\le 0$ are meant pointwise.

\begin{definition} \label{defn:GPP_and_generator}
Let $U$ be an on $(0,1)$ uniformly distributed random variable (rv) and let
$\bfZ=(Z_t)_{t\in[0,1]}\in C[0,1]$ be a stochastic process on the interval $[0,1]$ having
continuous sample paths. We require that $U$ and $\bfZ$ are independent and choose an
arbitrary constant $M<0$. Then
\begin{equation}\label{eqn:definition_of_gpp}
\bfV:=(V_t)_{t\in[0,1]}:=\left(\max\left(-\frac U{Z_t},M\right)\right)_{t\in[0,1]}.
\end{equation}
defines a \emph{standard generalized Pareto process} (GPP) if $0\le Z_t\le m$, $E(Z_t)=1$,
$t\in[0,1]$, hold for some constant $m\ge 1$. A stochastic process $\bfZ\in C[0,1]$ with
these two properties will be called a \emph{generator}.
\end{definition}
The constant $M$ is incorporated in the above definition to ensure that $V_t>-\infty$ for
each $t\in[0,1]$. Note that the finite-dimensional marginal distributions of $\bfV$ provide
multivariate GPD with ultimately uniform tails; see, e.g., \citet{aulbf11}.

The process $\bfV$ is characterized by the fact that its \emph{functional} distribution function (df) is given by
\begin{equation*}
  P(\bfV\le f)
  = 1 - E\left(\sup_{t\in[0,1]}(\abs{f(t)}Z_t)\right),
  \quad f\in \bar E^-[0,1],\,\norm f_\infty\le x_0,
\end{equation*}
for some $x_0>0$. We set $\bar E^-[0,1]:=\{f\in E[0,1]:\,f\le 0\}$ where $E[0,1]$ denotes
the set of all bounded functions $f:[0,1]\to\R$ that have only a finite number of
discontinuities. Then
\begin{equation*}
\norm f_D:=E\left(\sup_{t\in[0,1]}(\abs{f(t)}Z_t)\right),\qquad f\in E[0,1],
\end{equation*}
defines a \emph{$D$-norm} on $E[0,1]$ with generator $\bfZ$; see \citet{aulfaho11}.
This representation of the df of $\bfV$ in terms of a $D$-norm is in complete analogy with
the multivariate case of a GPD. We refer to \citet[Section 5.1]{fahure10}.

It was established by \citet{aulfaho11} that a stochastic process $\bfX\in C[0,1]$ is in the
\emph{functional domain of attraction} of a max-stable process (MSP) $\bfxi$, denoted by
$\bfX\in\mathcal D(\bfxi)$, if and only if this is true for the univariate margins together with
convergence of the corresponding copula processes. A stochastic process
$\bfU=(U_t)_{t\in[0,1]}\in C[0,1]$ is called \emph{copula process} if $U_t$ is for each
$t\in[0,1]$ uniformly distributed on $(0,1)$.

For each standard GPP there is a corresponding \emph{standard} MSP, i.e., a stochastic
process $\bfeta=(\eta_t)_{t\in[0,1]}\in C[0,1]$ such that
\begin{equation} \label{eqn:distribution_function_of_standard_evp}
P(\bfeta\le f)=\exp\left(-\norm f_D\right),\qquad f\in \barE.
\end{equation}
Note that this implies $P(\eta_t\le x)=\exp(x)$, $x\le 0$, $t\in\unit$. On the other hand, the
df of each max-stable process $\bfeta$ having standard negative exponential margins has a
representation as in Equation~\eqref{eqn:distribution_function_of_standard_evp}; we refer
to \citet{aulfaho11} for details.

\subsection{Overview of the current paper}

We replace the rv $U$ in Equation~\eqref{eqn:definition_of_gpp} by a rv $W\ge 0$ which is
independent of $\bfZ$, too. However, the distribution of $W$ is different from the uniform
one and, thus, the process
\begin{equation} \label{eqn:definition_of_delta_neighborhood}
  \bfX:=(X_t)_{t\in[0,1]}:=\left(\max\left(-\frac W{Z_t}, M\right)\right)_{t\in[0,1]}
\end{equation}
is no longer a standard GPP.

This gives rise to the following problem: Based on the exceedances in a sample of $n$
independent copies $\bfX^{(1)},\dots,\bfX^{(n)}$ of $\bfX$ above a high threshold line,
how close can the df of $W$ get to that of $U$ with the difference still being detected? As
we consider exceedances above a high threshold, only the lower end of the df of $W$
matters. In other words, the problem suggests itself to define parametric models
$\set{H_\vartheta:\,\vartheta\in\Theta}$ for the df $H_\vartheta$ of $W$, such that we
can derive optimal tests detecting the deviation of the distribution of the upper tail of
$\bfX$ from that of $\bfV$, i.e., the deviation of $\vartheta$ from zero. This is the content
of the present paper, which is organized as follows.

In Section \ref{sec:delta-neighborhoods} we require that the df $H_\vartheta$ of $W$ has
a density $h_\vartheta$ near zero, which satisfies for some $\delta>0$ the expansion
\begin{equation} \label{eqn:first_expansion_of_density_of_H}
h_\vartheta(u)=1+\vartheta u^{\delta} + o(u^\delta) \qquad \mbox{as }u\downarrow 0
\end{equation}
with some parameter $\vartheta\in\Theta$, where zero is an inner
point of $\Theta\subset\R$. The standard exponential df, for
instance, satisfies this condition with $\delta=1$ and
$\vartheta=-1$. The null-hypothesis $\vartheta=0$ is meant to be
the uniform distribution on $(0,1)$.

In Section \ref{sec:exponential_family} we assume that the distribution of $W$ belongs to
an exponential family given by the probability densities on the interval
\begin{equation} \label{eqn:second_expansion_of_density_of_H}
h_\vartheta(u)=C(\vartheta)\exp(\vartheta T(u)),\qquad 0\le u\le 1,\, \vartheta\in\R.
\end{equation}

In both models we establish local asymptotic normality (LAN) of the point process of
exceedances among $\bfX^{(1)},\dots,\bfX^{(n)}$ above an increasing threshold line. The
results, which are stated in Theorem~\ref{thm:lan} and
Theorem~\ref{thm:lan_exponential_family}, provide in each model the corresponding central
sequence and, thus, optimal tests for testing $\vartheta=0$ against a sequence of
alternatives $\vartheta_n$ converging to zero as the sample size increases. It turns out
that the particular values of the exceedances contribute to the central sequence only in
model \eqref{eqn:first_expansion_of_density_of_H},
whereas in the exponential model \eqref{eqn:second_expansion_of_density_of_H} the
number of exceedances alone yields the central sequence. The fact that just the number of
realizations in  shrinking sets provides the central sequence was characterized for truncated
processes in quite a general framework in \citet{falk98} and \citet{fali98}.

As the central sequences and, thus, the asymptotically optimal tests also
depend on further parameters of the generator process $\bfZ$, which might be unknown in
practice, we consider an omnibus test for testing $\vartheta=0$ as well. We compute its
asymptotic relative efficiency (ARE) with respect to the optimal test in each model. While
ARE is positive in model \eqref{eqn:first_expansion_of_density_of_H}, it is zero in model
\eqref{eqn:second_expansion_of_density_of_H}.

\section{Testing in $\delta$-neighborhoods of a standard GPP} \label{sec:delta-neighborhoods}

This section deals with optimal tests in the model introduced in
\eqref{eqn:first_expansion_of_density_of_H}. We assume that the df of the rv $W\ge 0$ in
\eqref{eqn:definition_of_delta_neighborhood} belongs to a parametric family
$\set{H_\vartheta:\,\vartheta\in\Theta}$ of distributions, where $\Theta$ is an open
subset of $\R$ containing $0$. By $H_0$ we denote the uniform distribution on the interval
$(0,1)$. In addition to \eqref{eqn:first_expansion_of_density_of_H}, it is required that
there is some $u_0\in(0,1)$ such that the density $h_\vartheta(u)$ of $H_\vartheta(u)$
exists for $u\in[0,u_0]$, $\vartheta\in\Theta$, and satisfies for some $\delta\in(0,1]$ the
expansion
\begin{equation} \label{eqn:expansion_of_density_of_df_of_w}
h_\vartheta(u)=1+\vartheta u^\delta+ r_\vartheta(u),\qquad u\in[0,u_0],
\end{equation}
where $r_\vartheta(0)=0$, $\vartheta\in\Theta$, and
\begin{equation} \label{eqn:remainder_delta_neighborhood}
\sup_{0<\abs \vartheta\le \varepsilon_0} \abs{\frac{r_\vartheta(u)}{\vartheta u^{\delta}}}=o(1)
\end{equation}
 as $u\downarrow 0$ for some $\varepsilon_0>0$. Obviously,
\eqref{eqn:remainder_delta_neighborhood} is equivalent with $r_\vartheta(u)=o(\vartheta
u^\delta)$ as $u\downarrow0$, uniformly for $\abs\vartheta\le\eps_0$. Since we have
$h_0=1$ and $r_0=0$, \eqref{eqn:expansion_of_density_of_df_of_w} and
\eqref{eqn:remainder_delta_neighborhood} imply in particular the representation
\begin{equation*}
h_\vartheta(u)=h_0(u)\left(1+O\left(\left(H_0(u)\right)^\delta\right)\right),
\end{equation*}
i.e., the lower tail of $H_\vartheta$ is in a $\delta$-neighborhood of $H_0$; see
\citet[Section~2.2]{fahure10}.

Moreover, we assume
\begin{equation} \label{eqn:generator_bounded_away_from_zero}
  A:=E\left(\inf_{t\in[0,1]}Z_t\right)>0.
\end{equation}
As $\inf_{t\in[0,1]}Z_t\ge 0$, this condition is equivalent with the assumption that
$\inf_{t\in[0,1]}Z_t$ is not the constant function zero. \eqref{eqn:generator_bounded_away_from_zero} is,
for instance, satisfied if $\bfZ=2\bfU$, where $\bfU=\left(U_t\right)_{t\in[0,1]}\in C[0,1]$ is a
copula process such that $-\bfU\in\mathcal D(\bfeta)$, $\bfeta\in  C[0,1]$ being a
standard MSP. This is implied by the fact that in this case $P\left(\inf_{t\in[0,1]}U_t>0\right)=1$.

Note that \eqref{eqn:generator_bounded_away_from_zero} and H\"{o}lder's inequality also give
\begin{equation*}
B:=E\left(\inf_{t\in[0,1]}Z_t^{1+\delta}\right)>0.
\end{equation*}

\subsection{Local asymptotic normality}

In order to derive asymptotically optimal tests in this model, we first establish local
asymptotic normality (LAN) of the point process of exceedances
\begin{equation*}
N_{n,c}(B):=\sum_{i\le n}\varepsilon_{\sup_{t\in[0,1]}\left(X^{(i)}_t/c\right)}(B\cap
[0,1]),\qquad B\in\mathbb{B},
\end{equation*}
where $\bfX^{(i)}$, $i\le n$, are independent copies of $\bfX$ in
\eqref{eqn:definition_of_delta_neighborhood} and $c<0$. $\mathbb{B}$ denotes the
$\sigma$-field of Borel sets of $\R$ and $\varepsilon_x$ is the point measure with mass one
at $x$. Note that
\begin{equation} \label{eqn:exceedance_X_c}
\sup_{t\in[0,1]}\frac{X_t}{c}\le u\iff \bfX\ge cu, \qquad u\in[0,1],
\end{equation}
i.e., the random point measure $N_{n,c}$ actually represents those observations among
$\bfX^{(1)},\dots,\bfX^{(n)}$ which exceed the constant threshold function $c$.

Denote those observations among $\sup_{t\in[0,1]}\left(X^{(i)}_t/c\right)$ with
$\sup_{t\in[0,1]}\left(X^{(i)}_t/c\right)\le 1$, $i\le n$, by
$Y_1,\dots,Y_{\tau(n)}$ in the order of their outcome. Then we have
\begin{equation*}
N_{n,c}(B)=\sum_{k\le \tau(n)}\varepsilon_{Y_k}(B),\qquad
B\in\mathbb{B}.
\end{equation*}
By Theorem 1.4.1 in \citet{reiss93} we may assume without loss of generality that
$Y_1,Y_2,\dots$ are independent copies of a rv $Y$ with df
\begin{equation*}
P_\vartheta(Y\le u)=\frac{P_\vartheta(\bfX\ge cu)}{P_\vartheta(\bfX\ge c)},\qquad 0\le u\le 1,
\end{equation*}
under parameter $\vartheta$, and that they are independent of the total number $\tau(n)$,
which is binomial $B\left(n,P_\vartheta\left(\bfX \ge c\right)\right)$-distributed.

In the next lemma we provide the density $f_{\vartheta,c}$ of $\sup_{t\in[0,1]}\left(X_t/c\right)$ and, thus, the
density of $Y$ under $\vartheta$, which is $f_{\vartheta,c}/P_\vartheta(\bfX\ge c)$. By
$P*Z$ we denote the distribution of a rv $Z$.

\begin{lemma} \label{lem:density_of_exceedances}
Suppose that the distribution of the rv $W$ in \eqref{eqn:definition_of_delta_neighborhood}
belongs to the family $\set{H_\vartheta:\,\vartheta\in\Theta}$. Then there is some
$c_0<0$ such that the density $f_{\vartheta,c}(u)$, with respect to the Lebesgue measure, of the rv
$\sup_{0\le t\le 1}\left(X_t/c\right)$ exists for $\vartheta\in\Theta$, $c\in[c_0,0)$, $u\in\unit$,
and it is given by
\begin{align*}
f_{\vartheta,c}(u)&=\abs c \int_0^m z\, h_\vartheta(\abs c z u)\,\left(P*\inf_{t\in[0,1]}Z_t\right)(dz).
\intertext{Furthermore there exists $\eps_0>0$ such that}
f_{\vartheta,c}(u)&=\abs c A +\vartheta \abs c^{1+\delta} B u^\delta+o\left(\vartheta \abs c^{1+\delta}\right)
\end{align*}
uniformly for $\abs\vartheta\le\varepsilon_0$ and $u\in[0,1]$ as $c\uparrow0$; note that
$f_{0,c}(u)=\abs c A$.
\end{lemma}

\begin{proof}
Let $m$ be given as in Definition~\ref{defn:GPP_and_generator} and $u_0,\eps_0,\delta$
be given as in Equation~\eqref{eqn:expansion_of_density_of_df_of_w}. Then we obtain for
$c_0>\max\set{M,-u_0/m}$, $\vartheta\in\Theta$, $c\in[c_0,0)$ and $u\in\unit$ by
conditioning on $\inf_{t\in\unit}Z_t=z$ and Fubini's theorem
\begin{align*}
P_\vartheta(\bfX\ge cu)
&=P_\vartheta\left(\max\left(-\frac W{Z_t},M\right)\ge cu,\,t\in[0,1]\right)\\
&=P_\vartheta(W\le \abs c u Z_t,\,t\in[0,1])\\
&=P_\vartheta\left(W\le \abs c u \inf_{t\in[0,1]}Z_t\right)\\
&=\int_0^m P_\vartheta(W\le \abs c zu )\, \left(P*\inf_{t\in[0,1]}Z_t\right)(dz)\\
&=\int_0^m H_\vartheta(\abs c z u)\left(P*\inf_{t\in[0,1]}Z_t\right)(dz)\\
&=\int_0^u\abs c \int_0^m z\, h_\vartheta(\abs c z x) \PinfZdz\,dx.
\end{align*}
This representation implies that $\sup_{t\in[0,1]}\left(X_t/c\right)$ has the density
\begin{align*}
f_{\vartheta,c}(u)
&=\abs c \int_0^m z\, h_\vartheta(\abs c z u) \left(P*\inf_{t\in[0,1]}Z_t\right)(dz)\\
&=\abs c \int_0^m z \left(1+\vartheta(\abs c z u)^\delta +  r_\vartheta(\abs c z u)\right)\, \left(P*\inf_{t\in[0,1]}Z_t\right)(dz)\\
&=\abs c E\left(\inf_{t\in[0,1]}Z_t\right) + \vartheta \abs c^{1+\delta} u^\delta E\left(\left(\inf_{t\in[0,1]}Z_t\right)^{1+\delta}\right)+ o\left(\vartheta\abs c^{1+\delta}\right)\\
&=\abs c A+\vartheta \abs c^{1+\delta}u^\delta B + o\left(\vartheta\abs c^{1+\delta}\right)
\end{align*}
as $c\uparrow 0$, uniformly for $\abs \vartheta\le\varepsilon_0$ and $u\in[0,1]$. (Note that
$r_\vartheta(\abs c zu) = o\left(\vartheta\abs{c}^\delta\right)$ uniformly for
$\abs \vartheta\le\varepsilon_0$, $u\in[0,1]$ and $z\in[0,m]$ as $c\uparrow0$.) As
$H_0(u)=u$, we have $h_0(u)=1$, $u\in [0,1]$, which completes the proof.
\end{proof}

If $c_0\le c<0$, we have $P_0(Y\le u)=u$, i.e., the distribution
$\mathcal{L}_{\vartheta,c}(Y)$ of $Y$ under $\vartheta$ is then dominated by
$\mathcal{L}_{0,c}(Y)$. Thus the distribution $\mathcal{L}_\vartheta(N_{n,c})$ of
$N_{n,c}$ under $\vartheta$ is, in this case, dominated by $\mathcal{L}_{0}(N_{n,c})$ --
see, e.g., Theorem 3.1.2 in \citet{reiss93} -- and we obtain from \citet[Example
3.1.2]{reiss93} the $\mathcal{L}_{0}(N_{n,c})$-density of
$\mathcal{L}_\vartheta(N_{n,c})$
\begin{equation*}
  \begin{aligned}
    \frac{d\mathcal{L}_\vartheta(N_{n,c})}{d\mathcal{L}_{0}(N_{n,c})}(\mu) 
    &=\left(\prod_{i=1}^{\mu(\unit)}\frac{f_{\vartheta,c}(y_i)}{f_{0,c}(y_i)} \frac{P_{0}(\bfX \ge c)}{P_\vartheta(\bfX \ge c)}\right)\times{} \\
    &\tab{=\Bigg(\prod_{i=1}^{\mu(\unit)}}
      \times \left(\frac{P_\vartheta(\bfX \ge c)}{P_{0}(\bfX \ge c)}\right)^{\mu(\unit)}
      \left(\frac{1-P_\vartheta(\bfX \ge c)}{1-P_{0}(\bfX \ge c)}\right)^{n-\mu(\unit)}
  \end{aligned}
\end{equation*}
where $\mu=\sum_{i=1}^{\mu(\unit)}\varepsilon_{y_i}$, $0\le y_1,\dots,y_{\mu(\unit)}\le
1$ and $\mu(\unit)\le n$. The loglikelihood ratio is, consequently,
\begin{equation} \label{eqn:loglikelihood_ratio}
\begin{aligned}
L_{n,c}(\vartheta\mid0)
&:=\log\left\{\frac{d\mathcal{L}_\vartheta(N_{n,c})}{d\mathcal{L}_{0}(N_{n,c})}(N_{n,c}) \right\} \\
&\tab*{:}=\sum_{i\le\tau(n)}\log\left(\frac{f_{\vartheta,c}(Y_i)}{f_{0,c}(Y_i)}\frac{P_{0}(\bfX \ge c)}{P_\vartheta(\bfX \ge c)} \right) +\tau(n)\log\left(\frac{P_\vartheta(\bfX \ge c)}{P_{0}(\bfX \ge c)}\right) \\
&\tab{:=}+(n-\tau(n))\log\left(\frac{1-P_\vartheta(\bfX \ge c)}{1-P_{0}(\bfX \ge c)} \right).
\end{aligned}
\end{equation}

We let in the sequel $c=c_n$ depend on the sample size $n$ with
$c_n\uparrow 0$ and, equally, $\vartheta=\vartheta_n$ with
$\vartheta_n\to0$ as $n\to\infty$. Precisely, we put with
arbitrary $\xi\in\R$
\begin{equation} \label{eqn:sequence_of_alternatives_delta}
\vartheta_n:=\vartheta_n(\xi):=\frac{\xi}{\left(n\abs{c_n}^{1+2\delta}\right)^{1/2}}.
\end{equation}
The following result finally proves the desired LAN property of $N_{n,c}$; it is a crucial tool
for deriving asymptotically optimal tests in the subsequent subsection.

\begin{theorem} \label{thm:lan}
Suppose that $c_n\uparrow 0$, $n \abs{c_n}^{1+2\delta}\to \infty$ as $n\to\infty$. Then
we obtain for $\vartheta_n$ as in \eqref{eqn:sequence_of_alternatives_delta} the
expansion
\begin{align*}
L_{n,c_n}(\vartheta_n\mid 0)&= \frac{\xi B}{(1+\delta) A^{1/2}}(Z_{n1}+Z_{n2})- \frac{\xi^2 B^2}{2A(2\delta+1)}  +o_{P_{0}}(1)\\
&\to_{D_{0}} N\left(- \frac{\xi^2 B^2}{2A(2\delta+1)}, \frac{\xi^2 B^2}{A(2\delta+1)}\right)
\end{align*}
with
\begin{equation*}
Z_{n1}:=\frac{\tau(n)-n\abs{c_n} A}{(n\abs{c_n} A)^{1/2}}\to_{D_{0}} N(0,1),
\end{equation*}
and
\begin{equation*}
Z_{n2}:=\frac {1+\delta}{\tau(n)^{1/2}}\sum_{k\le\tau(n)}\left(Y_k^\delta-\frac 1{1+\delta}\right) \to_{D_{0}} N\left(0,\frac{\delta^2}{2\delta+1}\right)
\end{equation*}
being independent.
\end{theorem}

\begin{proof}
First we compile several facts that will be used in the proof. From Lemma  \ref{lem:density_of_exceedances} we obtain
\begin{equation*} \tag{Fact 1} \label{eqn:fact_1}
P_{0}(\bfX \ge c_n)=P(\bfV \ge c_n)=\abs{c_n} A
\end{equation*}
and, thus, a suitable version of the central limit theorem implies
\begin{equation*} \tag{Fact 2} \label{eqn:fact_2}
\frac{\tau(n)-n\abs{c_n} A}{(n\abs{c_n} A)^{1/2}}\to_{D_{0}} N(0,1).
\end{equation*}
Moreover, we conclude from Lemma~\ref{lem:density_of_exceedances} for $\abs
\vartheta\le\varepsilon_0$
\begin{equation*} \tag{Fact 3} \label{eqn:fact_3}
P_\vartheta(\bfX \ge c_n)-P_{0}(\bfX \ge c_n)
=\abs{c_n}^{1+\delta}\frac \vartheta{1+\delta} B+o\left(\vartheta\abs{c_n}^{1+\delta}\right)
\end{equation*}
and, thus,
\begin{equation*} \tag{Fact 4} \label{eqn:fact_4}
\frac{P_{\vartheta_n}(\bfX \ge c_n)- P_{0}(\bfX \ge c_n)}{P_{0}(\bfX \ge c_n)} =\frac 1{(n\abs{c_n})^{1/2}} \frac{\xi B}{(1+\delta) A}+ o\left(\frac 1{(n\abs{c_n})^{1/2}}\right).
\end{equation*}

Hence, Taylor expansion
$\log(1+\varepsilon)=\varepsilon-\varepsilon^2/2+o(\varepsilon^2)$ as $\varepsilon\to 0$
implies
\begin{align*}
&\tau(n)\log\left(\frac{P_{\vartheta_n}(\bfX \ge c_n)}{P_{0}(\bfX \ge c_n)}\right)
    + (n-\tau(n))\log\left(\frac{1-P_{\vartheta_n}(\bfX \ge c_n)}{1-P_{0}(\bfX \ge c_n)}\right) \\*
&=\tau(n)\Biggl\{\frac{P_{\vartheta_n}(\bfX \ge c_n)-P_{0}(\bfX \ge c_n)}{P_{0}(\bfX \ge c_n)}
    - \frac12 \left(\frac{P_{\vartheta_n}(\bfX \ge c_n)-P_{0}(\bfX \ge c_n)}{P_{0}(\bfX \ge c_n)}\right)^2 \\
&\tab{=\tau(n)\Biggl\{}
    + O\left(\abs{\frac{P_{\vartheta_n}(\bfX \ge c_n)-P_{0}(\bfX \ge c_n)}{P_{0}(\bfX \ge c_n)}}^3\right)\Biggr\} \\
&\tab{=}
    + (n-\tau(n))\biggl\{\frac{P_{0}(\bfX \ge c_n) - P_{\vartheta_n}(\bfX \ge c_n)}{1-P_{0}(\bfX \ge c_n)} \\
&\tab{= + (n-\tau(n))\biggl\{}
    + O\left(\abs{P_{\vartheta_n}(\bfX \ge c_n)-P_{0}(\bfX \ge c_n)}^2\right)\biggr\}\\
&= \tau(n)\frac{P_{\vartheta_n}(\bfX \ge c_n)-P_{0}(\bfX \ge c_n)}{P_{0}(\bfX \ge c_n)}
    + (n-\tau(n)) \frac{P_{0}(\bfX \ge c_n)-P_{\vartheta_n}(\bfX \ge c_n)}{1-P_{0}(\bfX \ge c_n)}\\
&\tab{=}
    - \frac{\tau(n)}2 \left( \frac{P_{\vartheta_n}(\bfX \ge c_n)-P_{0}(\bfX\ge c_n)} {P_{0}(\bfX \ge c_n)}\right)^2
    + o_{P_{0}}(1)
\end{align*}

as
\begin{align*}
\tau(n)\abs{\frac{P_{\vartheta_n}(\bfX \ge c_n) - P_{0}(\bfX \ge c_n)}{P_{0}(\bfX \ge c_n)}}^3
&\sim E_{0}(\tau(n))\,O\left(\frac 1{(n\abs{c_n})^{3/2}}\right) \\
&=nP_{0}(\bfX \ge c_n)\, O\left(\frac 1{(n\abs{c_n})^{3/2}}\right)
\to_{n\to\infty}0,
\end{align*}
where $\sim$ denotes asymptotic equivalence. The preceding convergence to zero follows
from the condition $n\abs{c_n}^{1+2\delta}\to_{n\to\infty}\infty$ and the equivalence
\begin{equation*}
(n-\tau(n))\abs{P_{\vartheta_n}(\bfX \ge c_n)-P_{0}(\bfX \ge c_n)}^2
\sim O\left(\abs{c_n}\right)
\to_{n\to\infty}0.
\end{equation*}

From the law of large numbers we obtain
\begin{align*}
&\tau(n)\left(\frac{P_{\vartheta_n}(\bfX \ge c_n)- P_{0}(\bfX \ge c_n)}{P_{0}(\bfX \ge c_n)}\right)^2\\
&\sim E(\tau(n))\left(\frac 1{(n\abs{c_n})^{1/2}}\frac{\xi B}{(1+\delta) A}+ o\left(\frac 1{(n\abs{c_n})^{1/2}}\right)\right)^2\\
&=n\abs{c_n} A \left(\frac 1{(n\abs{c_n})^{1/2}}\frac{\xi B}{(1+\delta) A}+ o\left(\frac 1{(n\abs{c_n})^{1/2}}\right)\right)^2\\
&\to_{n\to\infty} \frac{\xi^2 B^2}{(1+\delta)^2 A}.
\end{align*}
Moreover,
\begin{align*}
&\tau(n)
\frac{P_{\vartheta_n}(\bfX \ge c_n)-P_{0}(\bfX \ge c_n)}{P_{0}(\bfX \ge c_n)}
+(n-\tau(n)) \frac{P_{0} (\bfX \ge c_n)-P_{\vartheta_n}(\bfX \ge c_n)}{1-P_{0}(\bfX \ge c_n)}\\
&=\frac{P_{\vartheta_n}(\bfX\ge c_n)-P_{0}(\bfX\ge  c_n)}{P_{0}(\bfX\ge c_n)(1-P_{0}(\bfX\ge c_n))}
(\tau(n)-nP_{0}(\bfX> c_n))\\
&=(n \abs{c_n} A)^{1/2} \left(\frac 1{(n\abs{c_n})^{1/2}} \frac{\xi B}{(1+\delta) A(1+o(1))}+o\left(\frac 1{(n\abs{c_n})^{1/2}}\right)\right) \times{} \\
&\tab{=(}
    \times \frac{\tau(n)-nP_{0}(\bfX \ge c_n)}{(n \abs{c_n} A)^{1/2}}\\
&=\frac{\xi B}{(1+\delta) A^{1/2}} Z_{n1}+o_{P_{0}}(1).
\end{align*}
Altogether we have shown so far that
\begin{align*}
&\tau(n)\log\left(\frac{P_{\vartheta_n}(\bfX\ge c_n)}{P_0(\bfX\ge c_n)}\right) + (n-\tau(n))\log\left(\frac{1-P_{\vartheta_n}(\bfX\ge c_n)} {1-P_0(\bfX\ge c_n)}\right)\\
&=\frac{\xi B}{(1+\delta)A^{1/2}} Z_{n1} - \frac{\xi^2 B^2}{2(1+\delta)^2 A} + o_{P_0}(1).
\end{align*}

Next we show that
\begin{equation*}
\sum_{k\le\tau(n)}\log\left(\frac{f_{\vartheta_n,c_n}(Y_k)}{f_{0,c_n}(Y_k)} \frac{P_{0}(\bfX \ge c_n)}{P_{\vartheta_n}(\bfX \ge c_n)}\right)= \frac{\xi B}{A^{1/2} (1+\delta)} Z_{n2}-\frac{\xi^2 B^2\delta^2}{2A(2\delta+1)(1+\delta)^2}+o_{P_{0}}(1).
\end{equation*}
We have by Lemma \ref{lem:density_of_exceedances}
\begin{equation*}
\frac{f_{\vartheta_n,c_n}(Y_k)}{f_{0,c_n}(Y_k)} = 1 + \frac \xi{(n\abs{c_n})^{1/2}} \frac B A Y_k^\delta + r_{0}(Y_k,\vartheta_n,c_n),
\end{equation*}
where $r_{0}(Y_k,\vartheta_n,c_n)=o\left(\left(n\abs{c_n}\right)^{-1/2}\right)$ uniformly for $k$
and $n$ with
\begin{align*}
&E_{0}\left(r_{0}(Y_1,\vartheta_n,c_n)\right)\\
&=\int_0^1 r_{0}(t,\vartheta_n,c_n) \frac{f_{0,c_n}(t)}{P_{0}(\bfX \ge c_n)}\,dt\\
&= \int_0^1\left(\frac{f_{\vartheta_n,c_n}(t)}{f_{0,c_n}(t)} - 1 - \frac \xi{(n\abs{c_n})^{1/2}} \frac B A t^\delta\right) \frac{f_{0,c_n}(t)}{P_{0}(\bfX \ge c_n)}\,dt\\
&= \frac{P_{\vartheta_n}(\bfX \ge c_n) - P_{0}(\bfX \ge c_n)}{P_{0}(\bfX \ge c_n)}
   - \frac \xi{(n\abs{c_n})^{1/2}} \frac B A \int_0^1 t^\delta \,dt\\
&= \frac{P_{\vartheta_n}(\bfX \ge c_n) - P_{0}(\bfX \ge c_n)}{P_{0}(\bfX \ge c_n)}
   - \frac \xi{(n\abs{c_n})^{1/2}} \frac B {(1+\delta)A}
\end{align*}
and
\begin{equation*}
Var_{0}\left(r_{0}(Y_1,\vartheta_n,c_n)\right) \le E_{0}\left(r^2_{0}(Y_1,\vartheta_n,c_n)\right) = o\left(1/(n\abs{c_n})\right).
\end{equation*}
Using again the Taylor expansion
$\log(1+\varepsilon)=\varepsilon-\varepsilon^2/2+O(\varepsilon^3)$ as $\varepsilon\to 0$,
we deduce
\begin{align*}
&\sum_{k\le\tau(n)}\log\left(\frac{f_{\vartheta_n,c_n}(Y_k)}{f_{0,c_n}(Y_k)} \frac{P_{0}(\bfX \ge c_n)}{P_{\vartheta_n}(\bfX \ge c_n)}\right)\\
&=\sum_{k\le\tau(n)}\log\left(1+ \frac \xi{(n\abs{c_n})^{1/2}} \frac B A Y_k^\delta + r_{0}(Y_k,\vartheta_n,c_n)\right)\\
&\hspace*{.5cm}-\tau(n)\log\left(1 + \frac{P_{\vartheta_n}(\bfX \ge c_n) - P_{0}(\bfX \ge c_n)}{P_{0}(\bfX \ge c_n)}\right)\\
&= \sum_{k\le\tau(n)}\left(\frac \xi{(n\abs{c_n})^{1/2}} \frac B A Y_k^\delta + r_{0}(Y_k,\vartheta_n,c_n)- \frac{\xi^2}{2n\abs{c_n}} \frac{B^2}{A^2} Y_k^{2\delta}\right)\\
&\hspace*{.5cm}-\tau(n)\left( \frac{P_{\vartheta_n}(\bfX \ge c_n) - P_{0}(\bfX \ge c_n)}{P_{0}(\bfX \ge c_n)} - \frac 12 \left( \frac{P_{\vartheta_n}(\bfX \ge c_n) - P_{0}(\bfX \ge c_n)}{P_{0}(\bfX \ge c_n)}\right)^2\right)\\
&\hspace*{.5cm} + o_{P_{0}}(1)\\
&= \sum_{k\le\tau(n)}\left(\frac \xi{(n\abs{c_n})^{1/2}} \frac B A \left(Y_k^\delta -\frac 1{1+\delta}\right) + r_{0}(Y_k,\vartheta_n,c_n)- E_{0}(r_{0}(Y_1,\vartheta_n,c_n))\right)\\
&\hspace*{.5cm}- \frac{\xi^2}{2n\abs{c_n}} \frac{B^2}{A^2}\sum_{k\le\tau(n)} Y_k^{2\delta} + \frac{\tau(n)}{n\abs{c_n}} \frac{\xi^2B^2}{2(1+\delta)^2 A^2} + o_{P_{0}}(1)\\
&=\frac{\tau(n)^{1/2}}{(n\abs{c_n})^{1/2}} \frac 1{\tau(n)^{1/2}} \frac{\xi B} A \sum_{k\le\tau(n)}\left(Y_k^\delta - \frac 1{1+\delta}\right)\\
&\hspace*{.5cm}-\frac{\tau(n)}{2n\abs{c_n}}\frac{\xi^2 B^2}{A^2} \frac 1{\tau(n)} \sum_{k\le\tau(n)} Y_k^{2\delta} + \frac{\tau(n)}{n\abs{c_n}} \frac{\xi^2 B^2}{2(1+\delta)^2 A^2} + o_{P_{0}}(1)\\
&=\frac {\xi B} {A^{1/2}} \frac 1{\tau(n)^{1/2}} \sum_{k\le\tau(n)} \left(Y_k^\delta - \frac 1{1+\delta}\right) - \frac{\xi^2 B^2\delta^2}{2A (2\delta+1)(1+\delta)^2} + o_{P_{0}}(1)\\
&\to_{D_{0}} N\left(-\frac{\xi^2 B^2\delta^2}{2A (2\delta+1)(1+\delta)^2}, \frac{\xi^2 B^2\delta^2}{A (2\delta+1)(1+\delta)^2}\right)
\end{align*}
by the law of large numbers and the central limit theorem. This completes the proof.
\end{proof}

\subsection{Testing $\vartheta=0$ against $\vartheta=\vartheta_n$}

Denote by $u_\alpha=\Phi^{-1}(1-\alpha)$ the $(1-\alpha)$-quantile of the standard
normal df. By the Neyman-Pearson lemma and Theorem~\ref{thm:lan}, the test statistic
\begin{equation*}
\varphi_1\left(N_{n,c_n}\right)=1_{(u_\alpha,\infty)}\left(\frac{(2\delta+1)^{1/2}}{1+\delta} (Z_{n1}+Z_{n2})\right)
\end{equation*}
defines an asymptotically optimal level-$\alpha$ test, based on $N_{n,c_n}$, for $H_0:\;\vartheta=0$ against $\vartheta_n=\vartheta_n(\xi)=\xi/\left(n\abs{c_n}^{1+2\delta}\right)^{1/2}$ with $\xi>0$. As $\varphi_1\left(N_{n,c_n}\right)$ does not depend on $\xi$, this test is asymptotically optimal, uniformly in $\xi>0$.

The corresponding uniformly asymptotically optimal test for $H_0$ against $\vartheta_n(\xi)$ with $\xi<0$ is
\begin{equation*}
\varphi_2\left(N_{n,c_n}\right)=1_{(-\infty,-u_\alpha)}\left(\frac{(2\delta+1)^{1/2}}{1+\delta} (Z_{n1}+Z_{n2})\right).
\end{equation*}

The asymptotic power functions of these tests are provided by Theorem \ref{thm:lan} as
well. By LeCam's third lemma we obtain that under $\vartheta_n=\vartheta_n(\xi)$
\begin{align*}
L_{n,c_n}(\vartheta_n\mid 0)
&= \frac{\xi B}{(1+\delta) A^{1/2}}(Z_{n1}+Z_{n2})- \frac{\xi^2 B^2}{2A(2\delta+1)}  +o_{P_{\vartheta_n}}(1)\\
&\to_{D_{\vartheta_n}} N\left(\frac{\xi^2 B^2}{2A(2\delta+1)}, \frac{\xi^2 B^2}{A(2\delta+1)}\right)
\end{align*}
with
\begin{equation*}
Z_{n1}+Z_{n2} \to_{D_{\vartheta_n}} N\left(\frac{\xi B (1+\delta)}{A^{1/2}(2\delta+1)}, \frac{(1+\delta)^2}{2\delta+1}\right).
\end{equation*}

The asymptotic power functions of $\varphi_i$ are, consequently, given by
\begin{equation} \label{eqn:power_function_of_optimal test}
\lim_{n\to\infty}E_{\vartheta_n(\xi)}\left(\varphi_i\left(N_{n,c_n}\right)\right) = 1-\Phi\left(u_\alpha-\frac{\abs \xi B}{A^{1/2}(2\delta+1)^{1/2}}\right).
\end{equation}

A disadvantage of the optimal test statistics $\varphi_i\left(N_{n,c_n}\right)$ is the fact
that they require explicit knowledge of the constants $A$ and $\delta$. To overcome this
disadvantage, we consider in the following an alternative test.

Recall that the observations $Y_1,Y_2,\dots$ are independent and, under $\vartheta=0$,
uniformly on $\unit*$ distributed rv if the threshold $c$ is close to zero. Conditional on
the assumption that there is at least one exceedance, i.e., conditionally on $\tau(n)>0$,
the test statistic
\begin{equation*}
T_{n,c}:=\frac 1{\tau(n)^{1/2}} \sum_{k=1}^{\tau(n)} \Phi^{-1}(Y_k)
\end{equation*}
is under $H_0$ exactly $N(0,1)$-distributed. By $\Phi$ we denote the standard normal df. This
test statistic is analogous to that in \citet{falm09} for testing for a multivariate generalized
Pareto distribution.

The next result provides the asymptotic distribution of $T_{n,c_n}$ under the alternative
$\vartheta_n=\vartheta_n(\xi)$ as $n\to\infty$.

\begin{proposition} \label{prop:asymptotic_distribution_of_alternative_test_statistic}
Under the assumptions of Theorem \ref{thm:lan} we have
\begin{equation*}
T_{n,c_n}\to_{D_{\vartheta_n}} N\left(\xi \frac B {A^{1/2}} \int_{-\infty}^\infty x(\Phi(x))^\delta\varphi(x)\,dx,1\right).
\end{equation*}
\end{proposition}

\begin{proof}
First we compute the asymptotic mean and variance of $\Phi^{-1}(Y)$ under $\vartheta_n$
and $c_n$ for $n\to\infty$. From Lemma \ref{lem:density_of_exceedances} we obtain that
the density of $Y$ under $\vartheta_n$ is for $0\le u\le 1$ and $c_n\ge c_0$ given by
\begin{align*}
p_{\vartheta_n,c_n}(u)
&=\frac{f_{\vartheta_n,c_n}(u)}{P_{\vartheta_n}(\bfX\ge c_n)}\\
&=\frac{\abs{c_n}}{P_{\vartheta_n}(\bfX\ge c_n)}\int_0^m z h_{\vartheta_n}(\abs{c_n}uz)\, \left(P*\inf_{0\le t\le 1}Z_t\right)(dz).
\end{align*}
From Fubini's theorem and the substitution $u\mapsto \Phi(x)$ we, therefore, obtain
\begin{align*}
&E_{\vartheta_n, c_n}\left(\Phi^{-1}(Y)\right)\\
&=\int_0^1 \Phi^{-1}(u) p_{\vartheta_n,c_n}(u)\,du\\
&=\frac{\abs{c_n}}{P_{\vartheta_n}(\bfX\ge c_n)} \int_0^m z\int_0^1 \Phi^{-1}(u) h_{\vartheta_n}(\abs{c_n}u z)\,du\,\left(P*\inf_{0\le t\le 1}Z_t\right)(dz)\\
&=\frac{\abs{c_n}}{P_{\vartheta_n}(\bfX\ge c_n)} \int_0^m z\int_{-\infty}^\infty x h_{\vartheta_n}(\abs{c_n}\Phi(x)z)\varphi(x)\,dx\left(P*\inf_{0\le t\le 1}Z_t\right)(dz)
\end{align*}
where $\varphi(x)=\Phi'(x)=(2\pi)^{-1/2}\exp(-x^2/2)$, $x\in\R$, is the density of the
standard normal df $\Phi$.

From condition \eqref{eqn:expansion_of_density_of_df_of_w} we obtain the expansion
\begin{align*}
&\int_0^m z\int_{-\infty}^\infty x h_{\vartheta_n}(\abs{c_n}\Phi(x)z)\varphi(x)\,dx\,\left(P*\inf_{0\le t\le 1}Z_t\right)(dz)\\
&=\int_0^m z \int_{-\infty}^\infty  x\left(1+ \vartheta_n\left(\abs{c_n}\Phi(x) z\right)^\delta + r_{\vartheta_n}\left(\abs{c_n}\Phi(x) z\right)\right)\varphi(x)\,dx\left(P*\inf_{0\le t\le 1}Z_t\right)(dz)\\
&=\vartheta_n \abs{c_n} B \int_{-\infty}^\infty x \left(\Phi(x)\right)^\delta \varphi(x)\,dx + o\left(\vartheta_n\abs{c_n}^\delta\right).
\end{align*}

From  \ref{eqn:fact_1} and \ref{eqn:fact_3} we obtain
\begin{align*}
P_{\vartheta_n}(\bfX\ge c_n)&=P_0(\bfX\ge c_n)+ \abs{c_n}^{1+\delta} \frac{\vartheta_n}{1+\delta} + o\left(\vartheta_n\abs{c_n}^{1+\delta}\right)\\
&=\abs{c_n} A + \abs{c_n}^{1+\delta} \frac{\vartheta_n}{1+\delta} B +  o\left(\vartheta_n\abs{c_n}^{1+\delta}\right)
\end{align*}
and, thus,

\begin{align*}
&E_{\vartheta_n, c_n}\left(\Phi^{-1}(Y)\right)\\
&=\frac{\vartheta_n \abs{c_n}^\delta B \int_{-\infty}^\infty x\left(\Phi(x)\right)^\delta\varphi(x)\,dx + o\left(\vartheta_n\abs{c_n}^\delta\right)}
{A + \abs{c_n}^\delta \frac{\vartheta_n}{1+\delta} B + o\left(\vartheta_n\abs{c_n}^\delta\right)}\\
&=\frac{\frac \xi{(n\abs{c_n})^{1/2}} B \int_{-\infty}^\infty x\left(\Phi(x)\right)^\delta\varphi(x)\,dx + o\left(\vartheta_n\abs{c_n}^\delta\right)}
{A + \abs{c_n}^\delta \frac{\vartheta_n}{1+\delta} B + o\left(\vartheta_n\abs{c_n}^\delta\right)}\\
&= \frac \xi{(n\abs{c_n})^{1/2}} \frac B A \int_{-\infty}^\infty x\left(\Phi(x)\right)^\delta\varphi(x)\,dx + o\left(\frac 1{(n\abs{c_n})^{1/2}}\right)
\end{align*}

Equally, we obtain
\begin{align*}
&E_{\vartheta_n, c_n}\left(\left(\Phi^{-1}(Y)\right)^2\right)\\
&=\frac{\abs{c_n}}{P_{\vartheta_n}(\bfX\ge c_n)} \int_0^m z\int_{-\infty}^\infty x^2 h_{\vartheta_n}(\abs{c_n}\Phi(x)z)\varphi(x)\,dx\left(P*\inf_{0\le t\le 1}Z_t\right)(dz)\\
&=\frac{\abs{c_n}}{P_{\vartheta_n}(\bfX\ge c_n)} \int_0^m z\int_{-\infty}^\infty x^2 \left(1+ \vartheta_n\left(\abs{c_n} \Phi(x) z\right)^\delta + r_{\vartheta_n}\left(\abs{c_n} \Phi(x) z\right)\right)\times{} \\
&\tab{=\frac{\abs{c_n}}{P_{\vartheta_n}(\bfX\ge c_n)} \int_0^m z\int_{-\infty}^\infty x^2 \Big(1+ \vartheta_n(\abs{c_n}}
    \times \varphi(x)\,dx\left(P*\inf_{0\le t\le 1}Z_t\right)(dz)\\
&\sim 1
\end{align*}
and, thus, the asymptotic variance of $\Phi^{-1}(Y)$ is under $\vartheta_n$ and $c_n$ for $n\to\infty$ equivalent to $1$.

Finally we have the expansion
\begin{align*}
E_{\vartheta_n,c_n}(\tau(n))&=nP_{\vartheta_n}(\bfX\ge c_n)\\
&= n\abs{c_n} A + (n\abs{c_n})^{1/2} \frac B{1+\delta} + o\left(n\abs{c_n}^{1+\delta}\right)\\
&= n\abs{c_n}A (1+o(1)).
\end{align*}

Now we can compute the asymptotic distribution of $T_{n,c_n}$ under $\vartheta_n$. We have
\begin{align*}
T_{n,c_n}&=\frac 1{\tau(n)^{1/2}} \sum_{k=1}^{\tau(n)}\Phi^{-1}(Y_k)\\
&=\frac 1{\tau(n)^{1/2}} \sum_{k=1}^{\tau(n)}\left(\Phi^{-1}(Y_k)- E_{\vartheta_n,c_n} \left(\Phi^{-1}(Y)\right)\right) + \tau(n)^{1/2} E_{\vartheta_n,c_n} \left(\Phi^{-1}(Y)\right),
\end{align*}
where the first term is by a suitable version of the central limit theorem asymptotically
standard normal distributed, and
\begin{align*}
\tau(n)^{1/2} E_{\vartheta_n,c_n} \left(\Phi^{-1}(Y)\right) &\sim E_{\vartheta_n,c_n}\left(\tau(n)^{1/2}\right) E_{\vartheta_n,c_n} \left(\Phi^{-1}(Y)\right)\\
&\sim ( n\abs{c_n}A)^{1/2}  \frac \xi{(n\abs{c_n})^{1/2}} \frac B A \int_{-\infty}^\infty x\left(\Phi(x)\right)^\delta\varphi(x)\,dx\\
&\sim \xi \frac B {A^{1/2}} \int_{-\infty}^\infty x\left(\Phi(x)\right)^\delta\varphi(x)\,dx,
\end{align*}
which completes the proof.
\end{proof}

From Proposition \ref{prop:asymptotic_distribution_of_alternative_test_statistic} we obtain that
\begin{equation*}
\varphi^*_1\left(N_{n,c_n}\right):= 1_{(u_\alpha,\infty)}\left(T_{n,c_n}\right), \quad
\varphi^*_2\left(N_{n,c_n}\right):= 1_{(-\infty,-u_\alpha)}\left(T_{n,c_n}\right)
\end{equation*}
are one-sided tests for testing $\vartheta>0$ and $\vartheta<0$, respectively, against $0$.
Their asymptotic power functions are given by
\begin{equation} \label{eqn:power_function_of_omnibus_test}
\begin{aligned}
\beta(\xi)&:=\lim_{n\to\infty}E_{\vartheta_n(\xi)}\left(\varphi^*_i\left(N_{n,c_n} \right)\right) \\
&=1-\Phi\left(u_\alpha-\abs\xi \frac B {A^{1/2}}\int_{-\infty}^\infty x (\Phi(x))^\delta \varphi(x)\,dx\right),\qquad \xi\in\R.
\end{aligned}
\end{equation}

The asymptotic relative efficiency of $\varphi^*_i\left(N_{n,c_n}\right)$ with respect to $\varphi_i\left(N_{n,c_n}\right)$ is, by \eqref{eqn:power_function_of_optimal test} and \eqref{eqn:power_function_of_omnibus_test}, given by the ratio
\begin{equation*}
\frac{\left(\abs\xi B\int_{-\infty}^\infty x (\Phi(x))^\delta \varphi(x)\,dx/A^{1/2}\right)^2}
{\left(\abs\xi B/\left(A^{1/2}(2\delta+1)^{1/2}\right)\right)^2}
= (2\delta+1)\left(\int_{-\infty}^\infty x (\Phi(x))^\delta \varphi(x)\,dx\right)^2,
\end{equation*}
which is independent of $\xi$.

Denote by
$k_n:=\min\set{k\in\N:\,E_{\vartheta_n(\xi)}\left(\varphi_i^*\left(N_{k,c_k}\right)\right)\ge
E_{\vartheta_n(\xi)}\left(\varphi_i\left(N_{n,c_n}\right)\right)}$ the least sample size, for which
$\varphi_i^*\left(N_{k_n,c_{k_n}}\right)$ is, at $\vartheta_n(\xi)$, at least as good as
$\varphi_i\left(N_{n,c_n}\right)$. The \emph{relative efficiency} of
$\varphi_i^*\left(N_{k_n,c_{k_n}}\right)$ with respect to $\varphi_i\left(N_{n,c_n}\right)$
is then defined as $n/k_n$. From \eqref{eqn:power_function_of_optimal test} and
\eqref{eqn:power_function_of_omnibus_test} we obtain that
\begin{equation} \label{eqn:asymptotic_relative_efficiency}
\lim_{n\to\infty} \frac{n \abs{c_n}^{1+2\delta}}{k_n \abs{c_{k_n}}^{1+2\delta}}= (2\delta+1)\left(\int_{-\infty}^\infty x (\Phi(x))^\delta \varphi(x)\,dx\right)^2,
\end{equation}
see Section 10.2 in \citet{pfanz94} for the underlying reasoning. This explains the
significance of the asymptotic relative efficiency defined above.

\begin{figure}
  \input{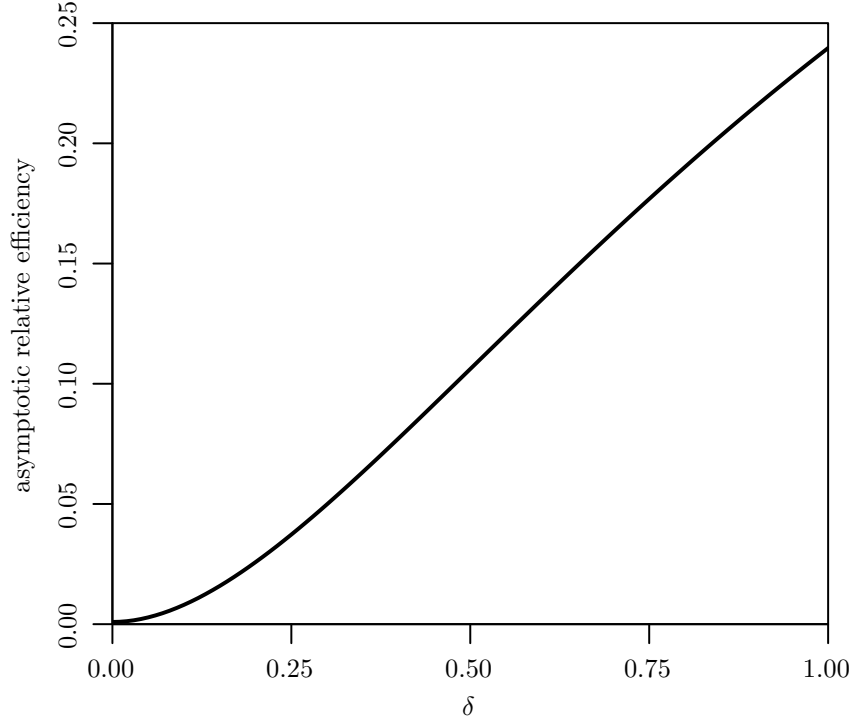}
  \caption{Asymptotic relative efficiency as in \eqref{eqn:asymptotic_relative_efficiency}.}
\end{figure}

\section{Testing in an exponential family model} \label{sec:exponential_family}

In this section we assume that the distribution of $W$ belongs to an exponential family
given by the probability densities on the interval $\unit$
\begin{equation*}
h_\vartheta(u)=C(\vartheta) \exp(\vartheta T(u)),\qquad 0\le u\le 1,\,\vartheta\in\R,
\end{equation*}
where $T:[0,1]\to\R$ is a bounded Borel-measurable function satisfying
\begin{equation*} 
\lim_{u\downarrow 0} T(u)=:C\in\R,
\end{equation*}
and $C(\vartheta)$ is defined by
\begin{equation*}
C(\vartheta):=\frac 1 {\int_0^1\exp(\vartheta T(u))\,du},\qquad \vartheta\in\R.
\end{equation*}

\begin{remark} \label{rem:density_of_supX_t_c_for_exponential_family} 
From the arguments in the proof of Lemma \ref{lem:density_of_exceedances} we obtain
that the rv $\sup_{0\le t\le 1}\left(X_t/c\right)$ has for $c<0$ close to zero and each $\vartheta\in\R$
on $\unit$ the Lebesgue-density
\begin{equation*}
f_{\vartheta,c}(u)=\abs c\int_0^mz h_\vartheta(\abs c zu) {\PinfZdz},\qquad 0\le u\le 1.
\end{equation*}
\end{remark}

In what follows we put with arbitrary $\xi\in\R$
\begin{equation*}
\vartheta_n:=\vartheta_n(\xi):=\frac \xi{(n\abs{c_n})^{1/2} A^{1/2} \left(C - \int_0^1T(u)\,du\right)},
\end{equation*}
where we require that $C\not=\int_0^1T(u)\,du$.

\begin{theorem} \label{thm:lan_exponential_family}
Suppose that $\abs{c_n}\to 0$, $n\abs{c_n}\to\infty$ as $n\to\infty$. Then we obtain for
the loglikelihood ratio in \eqref{eqn:loglikelihood_ratio} the expansion
\begin{equation*}
L_{n,c_n}(\vartheta_n\mid 0)=\xi Z_{n1}-\frac {\xi^2}2 + o_{P_0}(1).
\end{equation*}
\end{theorem}

\begin{proof}
Again we compile several facts first.
\begin{equation*} \tag{Fact 5} \label{eqn:fact_5}
C(\vartheta_n)=1- \vartheta_n\int_0^1 T(u)\,du + o(\vartheta_n),\qquad n\in\N.
\end{equation*}

This follows from the expansion $\exp(x)=1+x+o(x)$ as $x\to 0$:
\begin{align*}
C(\vartheta_n)&=\frac 1 {\int_0^1\exp(\vartheta T(u))\,du}\\
&=\frac 1 {\int_0^1 1+\vartheta_n T(u) + o(\vartheta_n)\,du}\\
&= \frac 1 {1+\vartheta_n\int_0^1  T(u)\,du + o(\vartheta_n)}\\
&= 1 - \vartheta_n\int_0^1  T(u)\,du + o(\vartheta_n).
\end{align*}

\begin{align*}
&P_{\vartheta_n}(\bfX\ge c_n)- P_0(\bfX\ge c_n) \tag{Fact 6} \label{eqn:fact_6}\\
&=\vartheta_n\abs{c_n}A\left(C -\int_0^1 T(u)\,du\right)+ o(\vartheta_n\abs{c_n})\\
&=\left(\frac{\abs{c_n}} n\right)^{1/2}A^{1/2}\xi + o\left(\left(\frac{\abs{c_n}} n\right)^{1/2}\right).
\end{align*}
This can be seen as follows. From Remark
\ref{rem:density_of_supX_t_c_for_exponential_family} and \ref{eqn:fact_5} we obtain
\begin{align*}
&P_{\vartheta_n}(\bfX\ge c_n) - P_0(\bfX\ge c_n)\\
&= \int_0^1 f_{\vartheta_n,c_n}(u)\,du - \int_0^1 f_{0.c_n}(u)\,du\\
&=\abs{c_n} \int_0^m z \int_0^1 h_{\vartheta_n}(\abs{c_n} z u) - 1\,du \PinfZdz\\
&= \abs{c_n} \int_0^m z \int_0^1 C(\vartheta_n)\exp\left(\vartheta_n T(\abs{c_n}zu)\right) - 1\, du \PinfZdz\\
&= \abs{c_n} \int_0^m z \int_0^1\left(1-\vartheta_n\int_0^1 T(x)\,dx+ o(\vartheta_n)\right) (1+\vartheta_n C+o(\vartheta_n))- 1\,du\\
&\tab{= \abs{c_n} \int_0^m z \int_0^1\left(1-\vartheta_n\int_0^1 T(x)\,dx+ o(\vartheta_n)\right) (1+\vartheta}
    \PinfZdz\\
&=\abs{c_n}\vartheta_n\left(C - \int_0^1T(x)\,dx\right) A+ o(\abs{c_n}\vartheta_n),
\end{align*}
which is \ref{eqn:fact_6}.

\ref{eqn:fact_6} together with \ref{eqn:fact_1} yields
\begin{equation*} \tag{Fact 7} \label{eqn:fact_7}
\frac{P_{\vartheta_n}(\bfX\ge c_n)- P_0(\bfX\ge c_n)}{P_0(\bfX\ge c_n)}= \frac 1 {(n\abs{c_n})^{1/2}} \frac \xi{A^{1/2}} + o\left(\frac 1  {(n\abs{c_n})^{1/2}}\right).
\end{equation*}

Repeating the arguments in the proof of Theorem \ref{thm:lan} one shows that
\begin{align*}
&\tau(n) \log\left(\frac{P_{\vartheta_n}(\bfX\ge c_n)}{P_0(\bfX\ge c_n)}\right) + (n-\tau(n)) \log\left(\frac{1-P_{\vartheta_n}(\bfX\ge c_n)}{1-P_0(\bfX\ge c_n)}\right)\\
&=\xi \frac{\tau(n) - n P_0(\bfX\ge c_n)}{(n\abs{c_n}A)^{1/2}} - \frac{\xi^2}2 + o_{P_0}(1)\\ &=\xi Z_{n1} - \frac{\xi^2}2 + o_{P_0}(1).
\end{align*}

It remains to show that
\begin{equation} \label{eqn:observations_vanish_asymptotically}
\sum_{k\le\tau(n)}\log\left(\frac{f_{\vartheta_n,c_n}(Y_k)} {f_{0,c_n}(Y_k)} \frac{P_0(\bfX\ge c_n)} {P_{\vartheta_n}(\bfX\ge c_n)}\right) = o_{P_0}(1).
\end{equation}

Repeating the arguments in the proof of \ref{eqn:fact_6} we obtain
\begin{align*}
&\frac{f_{\vartheta_n,c_n}(u) - f_{0,c_n}(u)}{f_{0,c_n}(u)}\\
&=\frac 1 A \int_0^m z\left(h_\vartheta(\abs{c_n}zu) - 1\right)\PinfZdz\\
&=\frac 1 A \int_0^m z\left( C(\vartheta_n)\exp(\vartheta_n T(\abs{c_n} z u))- 1\right) \PinfZdz\\
&=\frac 1 A \int_0^m z\left(\vartheta_n C - \vartheta_n \int_0^1 T(x)\,dx + o(\vartheta_n)\right)\PinfZdz\\
&=O(\vartheta_n)
\end{align*}
uniformly for $u\in\unit$ and $n\in\N$. The expansion
$\log(1+\varepsilon)=\varepsilon-\varepsilon^2/2+ O\left(\varepsilon^2\right)$ for $\varepsilon\to
0$ together with \ref{eqn:fact_7}, thus, yields,
\begin{align*}
&\sum_{k\le\tau(n)}\log\left(\frac{f_{\vartheta_n,c_n}(Y_k)} {f_{0,c_n}(Y_k)} \frac{P_0(\bfX\ge c_n)} {P_{\vartheta_n}(\bfX\ge c_n)}\right)\\
&=\sum_{k\le\tau(n)}\log\left(1 + \frac{f_{\vartheta_n,c_n}(Y_k)-f_{0,c_n}(Y_k)}{f_{0,c_n}(Y_k)}  \right) \\
&\tab{=\sum}
    -\tau(n) \log\left(1 + \frac{P_{\vartheta_n}(\bfX\ge c_n) - P_0(\bfX\ge c_n)} {P_0(\bfX\ge c_n)}\right)\\
&=\sum_{k\le\tau(n)}\left( \frac{f_{\vartheta_n,c_n}(Y_k)-f_{0,c_n}(Y_k)}{f_{0,c_n}(Y_k)} - \frac12 \left( \frac{f_{\vartheta_n,c_n}(Y_k)-f_{0,c_n}(Y_k)}{f_{0,c_n}(Y_k)}\right)^2\right)\\
&\tab{=\sum}
    -\tau(n) \frac{P_{\vartheta_n}(\bfX\ge c_n) - P_0(\bfX\ge c_n)} {P_0(\bfX\ge c_n)} +\frac{\tau(n)}2 \left( \frac{P_{\vartheta_n}(\bfX\ge c_n) -P_0(\bfX\ge c_n)} {P_0(\bfX\ge c_n)}\right)^2\\
&\tab{=\sum}
    + O_{P_0}\left(\frac 1 {n\abs{c_n}^{1/2}}\right).
\end{align*}

Note that
\begin{equation*}
E_{P_0}\left(f_{\vartheta_n,c_n}(Y)\right) = \int_0^1 f_{\vartheta_n,c_n}(u)\,du=P_{\vartheta_n}(\bfX\ge c_n)
\end{equation*}
and
\begin{equation*}
f_{0,c_n}(u)=\abs{c_n} A = P_0(\bfX\ge c_n).
\end{equation*}
We, thus, obtain
\begin{align*}
&\sum_{k\le\tau(n)}\log\left(\frac{f_{\vartheta_n,c_n}(Y_k)} {f_{0,c_n}(Y_k)} \frac{P_0(\bfX\ge c_n)} {P_{\vartheta_n}(\bfX\ge c_n)}\right)\\
&= \sum_{k\le\tau(n)} \frac{f_{\vartheta_n,c_n}(Y_k) - P_{\vartheta_n}(\bfX\ge c_n)} {\abs{c_n} A} 
    - \frac 12 \sum_{k\le\tau(n)} \left(\frac{f_{\vartheta_n,c_n}(Y_k) - P_{\vartheta_n}(\bfX\ge c_n)} {\abs{c_n} A}\right)^2\\
&\tab{=}
    + \frac{\tau(n)}2 \left( \frac{P_{\vartheta_n}(\bfX\ge c_n) -P_0(\bfX\ge c_n)} {P_0(\bfX\ge c_n)}\right)^2 + O_{P_0}\left(\frac 1 {n\abs{c_n}^{1/2}}\right)\\
&=: \mathrm{I}_n -\mathrm{II}_n + \mathrm{III}_n + O_{P_0}\left(\frac 1 {n\abs{c_n}^{1/2}}\right).
\end{align*}

From \ref{eqn:fact_7} we obtain
\begin{equation} \label{eqn:expansion_of_III_n}
\mathrm{III}_n  \sim \frac{n\abs{c_n} A}2 \left( \frac 1{(n\abs{c_n})^{1/2}} \frac \xi{A^{1/2}} + o\left(\frac 1{(n\abs{c_n})^{1/2}}\right)\right)^2 \sim \frac{\xi^2} 2.
\end{equation}

Next we show that $\mathrm{I}_n =o_{P_0}(1)$. This assertion follows, if we show that
\begin{equation} \label{eqn:expansion_of_I_n}
E_{P_0}\left(\left(\frac {f_{\vartheta_n,c_n}(Y_k) - P_{\vartheta_n}(\bfX\ge c_n)} {\abs{c_n} A}\right)^2\right) = o\left(\frac1 {n\abs{c_n}}\right).
\end{equation}

By elementary arguments we obtain
\begin{align*}
&E_{P_0}\left(\left(\frac {f_{\vartheta_n,c_n}(Y_k) - P_{\vartheta_n}(\bfX\ge c_n)} {\abs{c_n} A}\right)^2\right)\\
&=\frac 1 {c_n^2A^2} E_{P_0}\left(\left( \int_0^m z \int_0^1 h_{\vartheta_n}(\abs{c_n} z Y)- h_{\vartheta_n}(\abs{c_n} z u) \PinfZdz\right)^2\right)\\
&= \frac {C(\vartheta_n)^2} {A^2}  E_{P_0}\left(\left( \int_0^m z \int_0^1 \exp(\vartheta_n T(\abs{c_n}z Y))- \exp(\vartheta_n T(\abs{c_n}z u))\,du\right.\right.\\
&\tab{= \frac {C(\vartheta_n)^2} {A^2}  E_{P_0}\Biggl(\biggl( \int_0^m z \int_0^1 \exp(\vartheta_n T(\abs{c_n}z Y))- \exp(}
    \PinfZdz\biggr)^2\Biggr)\\
&= o(\vartheta_n^2)
\end{align*}
which is \eqref{eqn:expansion_of_I_n}.

Finally we have
\begin{align*}
&E_{P_0}\left(\left(\frac{f_{\vartheta_n,c_n}(Y)-\abs{c_n}A}{\abs{c_n}A}\right)^2\right)\\
&= \frac 1 {A^2} E_{P_0}\left(\left(\int_0^m z (h_{\vartheta_n}(\abs{c_n}zY) - 1)\PinfZdz\right)^2\right)\\
&= \frac 1 {A^2} E_{P_0}\left(\left(\int_0^m z (C(\vartheta_n)\exp(\vartheta_nT(\abs{c_n}zY)) - 1)\PinfZdz\right)^2\right)\\
&= \frac 1 {A^2} E_{P_0}\left(\left(\int_0^m z \left(\left(1-\vartheta_n \int_0^1 T(u)\,du + o(\vartheta_n)\right)(1+\vartheta_nT(\abs{c_n}z Y)+o(\vartheta_n))-1\right)\right.\right.\\
&\tab{= \frac 1 {A^2} E_{P_0}\Biggl(\biggl(\int_0^m z \biggl(\left(1-\vartheta_n \int_0^1 T(u)\,du + o(\vartheta_n)\right)(1+\vartheta}
    \PinfZdz\biggr)^2\Biggr)\\
&= \vartheta_n^2 \left(C-\int_0^1 T(u)\,du\right)^2 + o(\abs{c_n}^2\vartheta_n^2).
\end{align*}

The law of large numbers implies
\begin{equation*}
\mathrm{II}_n\to_{n\to\infty} -\frac{\xi^2}2
\end{equation*}
in probability, and, hence, \eqref{eqn:expansion_of_III_n} yields
\begin{equation*}
\mathrm{III}_n - \mathrm{II}_n = o_{P_0}(1).
\end{equation*}

We, thus, have established \eqref{eqn:observations_vanish_asymptotically}, which
completes the proof of Theorem \ref{thm:lan_exponential_family}.
\end{proof}

The test statistic
\begin{equation*}
\phi_1(N_{n,c_n}):=1_{(u_\alpha,\infty)}(Z_{n1})
\end{equation*}
defines by the Neyman-Pearson lemma an asymptotically optimal level-$\alpha$ test, based
on $N_{n,c_n}$, for the null-hypothesis $\vartheta=0$ against the sequence of alternatives
$\vartheta_n=\vartheta_n(\xi)=\xi/\left((n\abs{c_n})^{1/2} A^{1/2}\left(C-\int_0^1
T(u)\,du\right)\right)$ with $\xi>0$. As $\phi_1(N_{n,c_n})$ does not depend on $\xi$, this
test is asymptotically optimal uniformly in $\xi>0$.

The corresponding uniformly optimal test for $\vartheta=0$ against $\vartheta_n(\xi)$ with
$\vartheta<0$ is
\begin{equation*}
\phi_2(N_{n,c_n}):=1_{(-\infty,-u_\alpha)}(Z_{n1}).
\end{equation*}

From LeCam's third lemma we obtain that under $\vartheta_n=\vartheta_n(\xi)$
\begin{equation*}
L_{n,c_n}(\vartheta_n\mid 0) = \xi Z_{n1} - \frac{\xi^2}2 + o_{P_n}(1) \to_{D_{\vartheta_n}} N\left(\frac{\xi^2}2,\xi^2\right),
\end{equation*}
with
\begin{equation*}
Z_{n1} \to_{D_{\vartheta_n}} N(\xi,1).
\end{equation*}
The asymptotic power functions of $\phi_i$, $i=1,2$, are, consequently, given by
\begin{equation*}
\lim_{n\to\infty} E_{P_{\vartheta_n}}\left(\phi_i(N_{n,c_n})\right)=1-\Phi(u_\alpha-\abs\xi),\qquad i=1,2.
\end{equation*}

Next we compute the performance of the statistic $T_{n,c}=\tau(n)^{-1/2}
\sum_{k=1}^{\tau(n)} \Phi^{-1}(Y_k)$ for the testing problem $\vartheta=0$ against
$\vartheta_n(\xi)$.

\begin{lemma} \label{lem:mean_and_variance_of_omnibus_test_statistic}
We have
\begin{equation*}
E_{\vartheta_n,c_n}\left(\Phi^{-1}(Y)\right) = o(\vartheta_n),\quad Var_{\vartheta_n,c_n}\left(\Phi^{-1}(Y)\right)=1+o(\vartheta_n^2).
\end{equation*}
\end{lemma}

As $T_{n,c_n}\to_{D_{\vartheta_n,c_n}} N(0,1)$, Lemma
\ref{lem:mean_and_variance_of_omnibus_test_statistic} implies that the test statistic
$T_{n,c_n}$ is not capable to detect the alternative $\vartheta_n=\vartheta_n(\xi)$.

\begin{proof}
We have
\begin{align*}
&E_{\vartheta_n,c_n}\left(\Phi^{-1}(Y)\right)\\
&=\int_0^1 \Phi^{-1}(u) p_{\vartheta_n,c_n}(u)\,du\\
&=\frac{\abs{c_n}}{P_{\vartheta_n}(\bfX\ge c_n)}\int_0^m z \int_0^1 h_{\vartheta_n}(\abs{c_n}uz) \,du \PinfZdz\\
&=\frac{\abs{c_n}}{P_{\vartheta_n}(\bfX\ge c_n)}\int_0^m z \int_{-\infty}^\infty x h_{\vartheta_n}(\abs{c_n}\Phi(x)z)\varphi(x)\,dx \PinfZdz.
\end{align*}

\ref{eqn:fact_7} implies
\begin{align*}
&\int_0^m z \int_{-\infty}^\infty x h_{\vartheta_n}(\abs{c_n}\Phi(x)z)\varphi(x)\,dx \PinfZdz\\
&= \int_0^m z \int_{-\infty}^\infty x C(\vartheta_n)\exp\left(\vartheta_n T\left(\abs{c_n}\Phi(x)z\right)\right)\varphi(x) \,dx \PinfZdz\\
&= \int_0^m z \int_{-\infty}^\infty x\left(1-\vartheta_n \int_0^1 T(u)\,du + o(\vartheta_n)\right) (1+\vartheta_n C + o(\vartheta_n)) \varphi(x) \,dx\\
&\hspace*{8cm} \PinfZdz\\
&= \int_0^m z \int_{-\infty}^\infty x\left(1+\vartheta_n\left(C- \int_0^1 T(u)\,du\right)+ o(\vartheta_n)\right) \varphi(x)\,dx \PinfZdz\\
&= o(\vartheta_n).
\end{align*}

We have, moreover,
\begin{align*}
&E_{\vartheta_n,c_n}\left(\left(\Phi^{-1}(Y)\right)^2\right)\\
&=\frac{\abs{c_n}}{P_{\vartheta_n}(\bfX\ge c_n)} \int_0^m z \int_{-\infty}^\infty x^2 h_{\vartheta_n}(\abs{c_n}\Phi(x) z) \varphi(x) \,dx\PinfZdz,
\end{align*}
where
\begin{align*}
&\int_0^m z \int_{-\infty}^\infty x^2 h_{\vartheta_n}(\abs{c_n}\Phi(x) z) \varphi(x) \,dx\PinfZdz\\
&=A + \vartheta_n\left(C - \int_0^1 T(u)\,du\right) A + o(\vartheta_n).
\end{align*}

From \ref{eqn:fact_1} and \ref{eqn:fact_6} we obtain
\begin{equation*}
P_{\vartheta_n}(\bfX\ge c_n)= \abs{c_n} A + \left(\frac{\abs{c_n}} n\right)^{1/2} A^{1/2} \xi + o\left(\left(\frac{\abs{c_n}} n\right)^{1/2}\right)
\end{equation*}
and, thus,
\begin{equation*}
E_{\vartheta_n,c_n}\left(\Phi^{-1}(Y)\right) = o(\vartheta_n)
\end{equation*}
and
\begin{align*}
Var_{\vartheta_n,c_n}\left(\Phi^{-1}(Y)\right)&= E_{\vartheta_n,c_n}\left(\left(\Phi^{-1}(Y)\right)^2\right) - E_{\vartheta_n,c_n}\left(\Phi^{-1}(Y)\right)^2\\
&=1+o(\vartheta_n). \qedhere
\end{align*}

\end{proof}

\section*{Acknowledgements}

This article has benefited substantially from discussions during the \emph{4th International
Conference of the ERCIM WG on Computing \& Statistics} (ERCIM'11), 17-19 December,
2011, University of London. The first author is in particular grateful to the organizers of this
conference, who gave him the opportunity of giving a talk about an earlier version of this
paper.

The first author was supported by DFG Grant FA 262/4-1.

\bibliographystyle{enbib_arXiv}
\bibliography{testing_for_a_GPP}

\begin{thebibliography}{13}
\providecommand{\natexlab}[1]{#1}
\providecommand{\url}[1]{\texttt{#1}}
\providecommand{\urlprefix}{URL }
\providecommand{\href}[2]{#2}
\providecommand{\doi}[1]{doi:\discretionary{}{}{}\href{http://dx.doi.org/#1}{#1}}
\providecommand{\selectlanguage}[1]{\relax}
\providecommand{\bibinfo}[2]{#2}
\providecommand{\eprint}[2][]{\href{#1}{#2}}

\bibitem[{Aulbach et~al.(2012{\natexlab{\textit{a}}})Aulbach, Bayer, and
  Falk}]{aulbf11}
\bibinfo{author}{\textsc{Aulbach, S.}}, \bibinfo{author}{\textsc{Bayer, V.}},
  and \bibinfo{author}{\textsc{Falk, M.}}
  (\bibinfo{year}{2012}{\natexlab{\textit{a}}}).
\newblock \bibinfo{title}{A multivariate piecing-together approach with an
  application to operational loss data}.
\newblock \textit{\bibinfo{journal}{Bernoulli}}.
\newblock \bibinfo{note}{To appear}.

\bibitem[{Aulbach et~al.(2012{\natexlab{\textit{b}}})Aulbach, Falk, and
  Hofmann}]{aulfaho11b}
\bibinfo{author}{\textsc{Aulbach, S.}}, \bibinfo{author}{\textsc{Falk, M.}},
  and \bibinfo{author}{\textsc{Hofmann, M.}}
  (\bibinfo{year}{2012}{\natexlab{\textit{b}}}).
\newblock \bibinfo{title}{The multivariate piecing-together approach
  revisited}.
\newblock \bibinfo{type}{Tech. Rep.}, \bibinfo{institution}{University of
  W{\"u}rzburg}.
\newblock \bibinfo{note}{Submitted},
  \eprint[http://arxiv.org/abs/1108.0920]{{\ttfamily arXiv:1108.0920
  [math.PR]}}.

\bibitem[{Aulbach et~al.(2012{\natexlab{\textit{c}}})Aulbach, Falk, and
  Hofmann}]{aulfaho11}
\bibinfo{author}{\textsc{Aulbach, S.}}, \bibinfo{author}{\textsc{Falk, M.}},
  and \bibinfo{author}{\textsc{Hofmann, M.}}
  (\bibinfo{year}{2012}{\natexlab{\textit{c}}}).
\newblock \bibinfo{title}{On max-stable processes and the functional
  {$D$}-norm}.
\newblock \bibinfo{type}{Tech. Rep.}, \bibinfo{institution}{University of
  W{\"u}rzburg}.
\newblock \bibinfo{note}{Submitted},
  \eprint[http://arxiv.org/abs/1107.5136]{{\ttfamily arXiv:1107.5136
  [math.PR]}}.

\bibitem[{Balkema and de~Haan(1974)}]{balh74}
\bibinfo{author}{\textsc{Balkema, A.~A.}}, and
  \bibinfo{author}{\textsc{de~Haan, L.}} (\bibinfo{year}{1974}).
\newblock \bibinfo{title}{Residual life time at great age}.
\newblock \textit{\bibinfo{journal}{Ann. Probab.}}
\newblock \textbf{\bibinfo{volume}{2}}, \bibinfo{pages}{792--804}.

\bibitem[{Buishand et~al.(2008)Buishand, de~Haan, and Zhou}]{buihz08}
\bibinfo{author}{\textsc{Buishand, T.~A.}}, \bibinfo{author}{\textsc{de~Haan,
  L.}}, and \bibinfo{author}{\textsc{Zhou, C.}} (\bibinfo{year}{2008}).
\newblock \bibinfo{title}{On spatial extremes: with application to a rainfall
  problem}.
\newblock \textit{\bibinfo{journal}{Ann. Appl. Stat.}}
\newblock \textbf{\bibinfo{volume}{2}}, \bibinfo{pages}{624--642}.

\bibitem[{Falk(1998)}]{falk98}
\bibinfo{author}{\textsc{Falk, M.}} (\bibinfo{year}{1998}).
\newblock \bibinfo{title}{Local asymptotic normality of truncated empirical
  processes}.
\newblock \textit{\bibinfo{journal}{Ann. Statist.}}
\newblock \textbf{\bibinfo{volume}{26}}, \bibinfo{pages}{692--718}.

\bibitem[{Falk et~al.(2010)Falk, H{\"u}sler, and Reiss}]{fahure10}
\bibinfo{author}{\textsc{Falk, M.}}, \bibinfo{author}{\textsc{H{\"u}sler, J.}},
  and \bibinfo{author}{\textsc{Reiss, R.-D.}} (\bibinfo{year}{2010}).
\newblock \textit{\bibinfo{title}{Laws of Small Numbers: Extremes and Rare
  Events}}.
\newblock \bibinfo{edition}{3rd} ed.
\newblock \bibinfo{publisher}{Birkh{\"a}user}, \bibinfo{address}{Basel}.

\bibitem[{Falk and Liese(1998)}]{fali98}
\bibinfo{author}{\textsc{Falk, M.}}, and \bibinfo{author}{\textsc{Liese, F.}}
  (\bibinfo{year}{1998}).
\newblock \bibinfo{title}{Lan of thinned empirical processes with an
  application to fuzzy set density estimation}.
\newblock \textit{\bibinfo{journal}{Extremes}}
\newblock \textbf{\bibinfo{volume}{1}}, \bibinfo{pages}{323--349}.

\bibitem[{Falk and Michel(2009)}]{falm09}
\bibinfo{author}{\textsc{Falk, M.}}, and \bibinfo{author}{\textsc{Michel, R.}}
  (\bibinfo{year}{2009}).
\newblock \bibinfo{title}{Testing for a multivariate generalized {P}areto
  distribution}.
\newblock \textit{\bibinfo{journal}{Extremes}}
\newblock \textbf{\bibinfo{volume}{12}}, \bibinfo{pages}{33--51}.

\bibitem[{de~Haan and Ferreira(2006)}]{dehaf06}
\bibinfo{author}{\textsc{de~Haan, L.}}, and \bibinfo{author}{\textsc{Ferreira,
  A.}} (\bibinfo{year}{2006}).
\newblock \textit{\bibinfo{title}{Extreme Value Theory: An Introduction}}.
\newblock Springer Series in Operations Research and Financial Engineering.
\newblock \bibinfo{publisher}{Springer}, \bibinfo{address}{New York}.

\bibitem[{Pfanzagl(1994)}]{pfanz94}
\bibinfo{author}{\textsc{Pfanzagl, J.}} (\bibinfo{year}{1994}).
\newblock \textit{\bibinfo{title}{Parametric Statistical Theory}}.
\newblock
\newblock \bibinfo{publisher}{De Gruyter}, \bibinfo{address}{Berlin}.

\bibitem[{Pickands(1975)}]{pick75}
\bibinfo{author}{\textsc{Pickands, J., III}} (\bibinfo{year}{1975}).
\newblock \bibinfo{title}{Statistical inference using extreme order
  statistics}.
\newblock \textit{\bibinfo{journal}{Ann. Statist.}}
\newblock \textbf{\bibinfo{volume}{3}}, \bibinfo{pages}{119--131}.

\bibitem[{Reiss(1993)}]{reiss93}
\bibinfo{author}{\textsc{Reiss, R.-D.}} (\bibinfo{year}{1993}).
\newblock \textit{\bibinfo{title}{A Course on Point Processes}}.
\newblock
\newblock \bibinfo{publisher}{Springer}, \bibinfo{address}{New York}.

\end{thebibliography}

\end{document}